\documentclass[12pt]{amsart}

\usepackage{amscd}
\usepackage{amsmath}
\usepackage{amsfonts}
\usepackage{amssymb}
\usepackage{graphicx}
\usepackage[T2A]{fontenc}
\usepackage[russian,english]{babel}
\usepackage[all]{xy}

\theoremstyle{plain}
\newtheorem{theorem}{Theorem}[section]
\newtheorem{lemma}{Lemma}[section]

\theoremstyle{definition}

\theoremstyle{definition}
\newtoks{\thehExample}
\newtheorem*{Example}{\the\thehExample}

\begin{document}
\selectlanguage{english}
\title{ Embedding  Semigroup $C^*$-algebras \\
into Inductive Limits}

\author{E.~V.~Lipacheva}
\email{elipacheva@gmail.com}
\address{Chair of Higher Mathematics,
Kazan State Power Engineering University, Krasnoselskaya 51,  Kazan,
Russian Federation, 420066}

\maketitle

\begin{abstract} 
The note
is concerned with  inductive systems of  Toeplitz algebras and their $*$-homomorphisms over arbitrary
partially ordered sets. The Toeplitz algebra is the  reduced semigroup $C^*$-algebra
for  the additive semigroup of non-negative integers. It is known that every partially ordered set can be
represented as the union of  the  family of its maximal upward directed subsets indexed by elements of some set. In our previous work we have
 studied a topology on this set of indexes. For every maximal upward directed subset we consider the inductive system of Toeplitz algebras that is defined by a given inductive system over an arbitrary partially ordered set and its inductive limit. Then for a  base neighbourhood $U_a$
 of the  topology on the set of indexes we construct the $C^*$-algebra $\mathfrak{B}_a$ which is the direct product of those inductive limits.
In this note we continue studying  the connection between the properties of the topology on the set of indexes and properties of inductive limits for systems consisting of $C^*$-algebras $\mathfrak{B}_a$ and their $*$-homorphisms. It is proved that there exists an embedding of the  reduced semigroup $C^*$-algebra
for  a semigroup in the additive group of all rational numbers into the inductive limit for the system of $C^*$-algebras $\mathfrak{B}_a$.
\end{abstract}

\

\emph{keywords:}{embedding, inductive limit, inductive system,
injective $*$-homomorphism, partially ordered set,  reduced
semigroup $C^*$-algebra, semigroup, Toeplitz algebra, upward
directed set}

\section{Introduction}
In algebraic quantum field theory, which serves as our main motivation, one usually
considers inductive systems of $C^*$-algebras and their $*$-homomorphisms. The simplest example
of such a system is a net of $C^*$-algebras and their embeddings.

In  \cite{Ruzzi2005, RuzziVasselli2012, Vasselli2015} the authors study nets containing
$C^*$-algebras of quantum observables for the case of curved
spacetimes.  The net  constructed by means of the semigroup
$C^*$-algebra generated by the path semigroup for a partially
ordered set is treated in \cite{GLS2016}.  In \cite{GLS2018} the
authors deal with a net consisting of $C^*$-algebras associated to a
net  of Hilbert spaces over a partially ordered set.

A part of motivation for studying inductive systems of Toeplitz
algebras comes from \cite{gumerov2018, GLG2018, GLG2018-2,
gumLJM2019}. The paper \cite{gumerov2018} contains results on limit
automorphisms  for inductive sequences of Toeplitz algebras which
are closely related to the facts on the mappings of topological
groups \cite{Gu05, gumerovLJM2005, Gu05s}. In \cite{GLG2018,
GLG2018-2} the authors introduce a topology associated to a
partially ordered set and study
 its relation to  properties of inductive limits arising from a system of $C^*$-algebras over the partially ordered set which yields that topology. The existence of an isomorphism between
 the inductive limit of an inductive system of Toeplitz algebras over a
directed set  and the reduced
semigroup  $C^*$-algebra for a semigroup in the group of rational
numbers is shown in \cite{GLG2018, gumLJM2019}.

 This note deals with inductive systems of Toeplitz algebras and their  $*$-homomorphisms over arbitrary partially ordered sets. Here, by the Toeplitz algebra we mean the reduced semigroup
$C^*$-algebra for the additive semigroup of non-negative integers.
The study
of such semigroup  $C^*$-algebras goes back to L.~A.~Coburn
\cite{Co67, coburn69}, R.~G.~Douglas \cite{douglas72},
G.~J.~Murphy \cite{murphy87, murphy91}. There is a large literature on the subject
 (see, for example,  \cite{murphy89, Liarxiv, LipSibMJ, LO15, AGL4} and
the references there in).

For a given
inductive system of Toeplitz algebras over a partially ordered set $K$ one can take an inductive subsystem and its inductive limit $\mathfrak{A}^{K_i}$ over every maximal upward directed subset $K_i$ in $K$, where $i$ runs over a set of indexes $I$. Using the topology on $I$ and the inductive limits $\mathfrak{A}^{K_i}$,  we construct new $C^*$-algebras $\mathfrak{B}_a$ that are the direct products of $C^*$-algebras $\mathfrak{A}^{K_i}$. Then we consider an inductive system consisting of  $C^*$-algebras $\mathfrak{B}_a$ over the upward directed set $K_i$ and the inductive limit $\mathfrak{B}^{K_i}$ of that system. We prove that the reduced semigroup $C^*$-algebra for  a semigroup in the additive group of all rational numbers can be embedded into $C^*$-algebra $\mathfrak{B}^{K_i}$. In other words, it is shown there exists an injective $*$-homomorphism between these $C^*$-algebras.

The present note  consists of  four sections. The first two sections are Introduction and
Preliminaries.  and three more sections containing the results.  Section~\ref{top}
 deals with the  topology on the index set $I$.  Section~\ref{res} contains an auxiliary statement and the main results about embedding the reduced semigroup $C^*$-algebra.

\section{Preliminaries}\label{preliminaries}
\label{sec:1}

Let $K$ be an upward directed set.
We shall consider the category associated to the set $K$,
which is denoted by the same letter \ $K$. We recall that the
objects of this category are the elements of the set \ $K$, and, for
any pair \ $a,b \in K$, the set of morphisms from \ $a$ \ to \ $b$ \
consists of the single element \ $(a,b)$ \ provided that \ $a\leq
b$, and is the void set otherwise.

Further, we consider a covariant functor \ $\mathcal{F}$ \ from the
category \ $K$ \ into the category of unital \ $C^*$-algebras and
their unital \ $*$-homomorphisms.
Such a functor is called \emph{an inductive system } in the category
of \ $C^*$-algebras over the set \ $(K, \, \leq~)$. It may be given
by a collection \ $(K,\{\mathfrak{A}_a\},\{\sigma_{ba}\})$ \
satisfying the properties from the definition of a functor.  We
shall write \ $\mathcal{F}=(K,\{\mathfrak{A}_a\},\{\sigma_{ba}\})$.
 Here, $\{\mathfrak{A}_a\mid a\in K\}$ is a family of unital
 $C^*$-algebras, and $\sigma_{ba}:\mathfrak{A}_a\longrightarrow \mathfrak{A}_b$,
 where $a\leq b$, are unital $*$-homomorphisms of
$C^*$-algebras. Recall that the equations
$\sigma_{ca}=\sigma_{cb}\circ\sigma_{ba}$ hold for all elements
$a,b,c \in K$ satisfying the condition $a\leq b\leq c$. Furthermore,
for each element $a\in K$ the morphism $\sigma_{aa}$ is the identity
mapping.


Throughout the paper, for a unital algebra $\mathfrak{A}$ its unit
will be denoted by $\mathbb{I}_{\mathfrak{A}}$.

 We recall the definition and  some facts concerning the inductive limits for inductive systems of $C^*$-algebras
(see, for example, \cite[Section 11.4]{KadisonRingrose},
\cite[Appendix L]{IL}).

\emph{The inductive limit} of the system
$\mathcal{F}=(K,\{\mathfrak{A}_a\},\{\sigma_{ba}\})$ is a pair
$(\mathfrak{A}^{K}, \{\sigma^{K}_a\})$ where
 \ $\mathfrak{A}^{K}$ is a $C^*$-algebra and $\{\sigma^{K}_a:\mathfrak{A}_a\rightarrow \mathfrak{A}^{K} \mid a\in
K\}$ is a family of canonical $*$-homomorphisms such that the
following diagram commutes whenever $a\leq b \ $:
\[
\xymatrix{
  \mathfrak{A}_a \ar[rr]^{\sigma_{ba}} \ar[dr]_{\sigma^{K}_a}
                &  &    \mathfrak{A}_b \ar[dl]^{\sigma^{K}_b}    \\
                & \mathfrak{A}^{K}                 }
\]
that is, the equality for mappings
\begin{equation}\label{sigma}
\sigma^{K}_a=\sigma^{K}_b\circ\sigma_{ba}
\end{equation}
holds. We note that one has the equality
\begin{equation}\label{cup}
\mathfrak{A}^{K}=\overline{\mathop\bigcup\limits_{a\in
K}\sigma^{K}_a(\mathfrak{A}_a)} \, ,
\end{equation}
where the bar means the closure of the set with respect to
the norm topology in the \ $C^*$-algebra \ $\mathfrak{A}^{K}$.

The inductive limit $(\mathfrak{A}^{K}, \{\sigma^{K}_a\})$  is
denoted as follows:
$$
(\mathfrak{A}^{K}, \{\sigma^{K}_a\}):=\varinjlim\mathcal{F}.
$$
The $C^*$-algebra $\mathfrak{A}^{K}$ itself is often called the
inductive limit.

  The exact construction  of the inductive limit  (\ref{cup})  and the explicit form
of the canonical $*$-homomorphisms are less important than the following
\emph{universal behavior}.

\begin{lemma}{\rm \cite[Appendix L, theorem L.1.1]{IL}}\label{UniversalProp}
Let $\mathfrak{B}$ be another $C^*$-algebra and
$\psi_a:\mathfrak{A}_a\longrightarrow\mathfrak{B}$ be a canonical
$*$-homomorphism for each $a\in K$, and the condition analogous to
(\ref{sigma}) is satisfied, that is,
$\psi_a=\psi_b\circ\sigma_{ba}$ for every $a\leq b \ $. Then the
following commutative diagram can be filled out with precisely one
$*$-homomorphism $\theta$ from $\mathfrak{A}^{K}$ onto
$\mathfrak{B}$ that leaves the diagram commutative:
\begin{equation}\label{theta}
\xymatrix{ \mathfrak{A}_a \ar[rr]^{\sigma_{ba}}
\ar[dr]^{\sigma^{K}_a} \ar[ddr]_{\psi_a} & & \mathfrak{A}_b \ar[dl]_{\sigma^{K}_b} \ar[ddl]^{\psi_b} \\
&  \mathfrak{A}^{K} \ar[d]^{\theta} & \\
&  \mathfrak{B}     }
\end{equation}
\end{lemma}

The next lemma is usefull when we try to find out whether $\theta$
is injective.

\begin{lemma}{\rm \cite[Appendix L, lemma L.1.3]{IL}}\label{injective}
Let the commutative diagram (\ref{theta}) be hold for all $a\leq b \
$. If each $\psi_a:\mathfrak{A}_a\longrightarrow\mathfrak{B}$ is
injective then $\theta$ is also injective.
\end{lemma}

Further, we recall the definition of the reduced semigroup \
$C^*$-algebras for semigroups in the group of
all rational numbers \ $\mathbb{Q}$.

Assume that \ $\Gamma$ \ is an arbitrary subgroup in  \ $\mathbb{Q}$. Let \ $\Gamma^+:=\Gamma\cap
[0,+\infty)$ \ be the positive cone in the ordered group \ $\Gamma.$
As usual, the symbol \ $l^2(\Gamma^+)$ \ stands for the Hilbert
space of all square summable complex-valued functions on the
additive subgroup $\Gamma^+$. The canonical orthonormal basis in the
Hilbert space \ $l^2(\Gamma^+)$ \ is denoted by \ $\{\, e_g \mid
g\in \Gamma^+ \, \}$. That is, for all elements \ $g,h\in \Gamma^+$,
 we set \ $e_g(h)=\delta_{g,h}$, where $\delta_{g,h}=1$ if $g=h$,
 and $\delta_{g,h}=0$ if $g\neq h$.

 Let us consider the $C^*$-algebra  of all bounded linear operators $B(l^2(\Gamma^+))$ in the Hilbert space $l^2(\Gamma^+)$.
For every element $g\in \Gamma^+$, we define the isometry $V_g\in
B(l^2(\Gamma^+))$ by
$$
V_ge_h=e_{g+h},
$$
 where $h\in \Gamma^+.$

  We denote by $C^*_r(\Gamma^+)$ the $C^*$-subalgebra in the algebra $B(l^2(\Gamma^+))$ generated
 by the set $\{V_g|g\in \Gamma^+\}$. It is called
  \textit{the reduced semigroup $C^*$-algebra of the semigroup $\Gamma^+$}, or
 \textit{the~Toeplitz algebra generated by $\Gamma^+$}.

In the case when \ $\Gamma$ \ is the group of all integers \
$\mathbb{Z}$, we also denote the semigroup $C^*$-algebra \
$C^*_r(\mathbb{Z}^+)$ \ by $\mathcal{T}$ and use the symbols $T$ and
$T^n$ instead of $V_1$ and $V_n$, respectively, where $n\in
\mathbb{Z^+}$.

In the similar way a semigroup $C^*$-algebra can be defined for an
arbitrary cancellative semigroup. As is noted in
\cite[Section~2]{Liarxiv}, a semigroup $C^*$-algebra is a  very
natural object. It is generated by the left regular representation
of a given semigroup.

Let $P=(p_1,p_2,p_3,...)$ be an arbitrary sequence of prime numbers. In what follows we shall consider the reduced semigroup  $C^*$-algebra  $C^*_r(Q_P^+)$ for the semigroup
$$Q_P^+=\{\frac{m}{p_1\cdot p_2\cdot...\cdot p_n}\mid m\in \mathbb{Z}^+,n\in
\mathbb{N}\}$$
of rational numbers.

It follows from Coburn's theorem \cite[Theorem 3.5.18]{murphy} that for
every number \ $n\in \mathbb{N}$, there exists a unique isometric \
$*$-homo\-morphism of \ $C^*$-algebras \
$\varphi:\mathcal{T}\longrightarrow \mathcal{T}$ \ such that \ $
\varphi(T)=T^n$. We note that a straightforward proof of the existence of the homo\-morphism $\varphi$ is given in \cite[Proposition~3]{gumerovConf}.

Consequently, for every sequence of prime numbers
$P=(p_1,p_2,p_3,...)$ one can construct the inductive sequence of
Toeplitz algebras $(\{\mathcal{T}_n\},\{\varphi_{n,n+1}\}),$ where
$\mathcal{T}_n=\mathcal{T}$, and the bonding $*$-homomorphisms are
defined as follows:
$$\varphi_{n,n+1}:\mathcal{T}_n\longrightarrow\mathcal{T}_{n+1}:T\mapsto
T^{p_n}, \ n\in \mathbb{N}.$$ Let us denote by $\mathfrak{T}$ the inductive limit of this sequence.  The following statement is proved in \cite{gumerov2018}.

\begin{lemma}{\rm \cite[Proposition 1]{gumerov2018}}\label{gumerov}
There exists an isomorphism of $C^*$-algebras:
$$\mathfrak{T}\simeq C^*_r(Q_p^+).
$$
\end{lemma}

For additional results in the theory of $C^*$-algebras we refer the
reader, for example, to \cite{blackadar}; \cite[Ch.\,4,
\S\,7]{helemskii} and \cite{murphy}. Necessary facts from the theory
of categories and functors are contained, for example, in \cite[
Ch.\,0, \S\,2]{helemskii} and \cite{BucurDeleanu1968}.

\section{Topology on an index set}\label{top}

Throughout the next sections we shall consider an arbitrary partially ordered
set \ $(\,K, \, \leq \,)$ that is not necessarily directed.
Taking the family of all upward directed subsets of the set \
$(\,K, \, \leq \,)$ and  using \,  Zorn's lemma, one can
easily prove the following statement.

\begin{lemma}\label{Kdecomposition}
Let \ $(\,K, \, \leq \,)$ \ be a partially ordered set. Then the
following equality holds:
\begin{equation}\label{KbigcupKi}
K=\bigcup\limits_{i\in I}K_i,
\end{equation}
where \ $\left \{ \, K_i \, | \, i\in I \, \right \}$ \ is the
family of all maximal upward directed subsets of  \, $(K, \, \leq
)$.

\end{lemma}

We consider the topology on the index set $I$ which was introduced in \cite{GLG2018,GLG2018-2}.  For the convenience of the reader we recall the  definition of this topology and its properties.

 For every element $a\in K$ we define the set
$U_a=\{i\in I \ : \ a\in K_i\} $. The family of sets $\{U_a \mid
a\in K\}$ satisfies the following properties:
\begin{itemize}
    \item[$-$] If $a,b\in K$ such that $a\leq b$ then $U_b\subset U_a$.
    \item[$-$] The family $\{U_a \mid a\in K\}$ is a base for a topology on the set $I$.
\end{itemize}

We denote by $\tau$ the topology generated by the base $\{U_a \mid
a\in K\}$. The  topological space $(I,\tau)$ is a $T_1$-space.

Examples of different topological spaces $(I,\tau)$ are contained in \cite{GLG2018,GLG2018-2}. In particular, $(I,\tau)$ may be: a non-Hausdorff space \cite[Example~1]{GLG2018-2}, a locally compact space \cite[Example~2]{GLG2018-2}, a discrete space \cite[Example~3]{GLG2018-2}. Here we give an example of a compact space.

\textbf{Example.} As the set $K$ we consider the set of all closed
arcs in the unit circle $S^1$:
\begin{multline*}
K:=\left\{A\subset
S^1 \ \big | \ A=[e^{2\pi ix},e^{2\pi iy}] \
\mbox{\rm or } \ A=S^1\setminus(e^{2\pi ix},e^{2\pi iy}), \right. \\
\left. \mbox{\rm where } \ x,y\in [0,1) \ \mbox{\rm and } \ x<y
\right\}.
\end{multline*}
A partial order on $K$ is defined in the following way: for $A,B\in
K$ we put $A\leq B \Leftrightarrow A\subset B$.

It is easily verified that the pair $(K,\leq)$ is a partially
ordered set. Moreover, it is worth noting that this set is not
directed. Indeed, take any $x_1,x_2,x_3,x_4\in[0,1)$ such that
$x_1<x_2<x_3<x_4$. Then for $A=[e^{2\pi ix_1},e^{2\pi ix_4}]$ and
$B=S^1\setminus(e^{2\pi ix_2},e^{2\pi ix_3})$ there is not $C\in K$
such that $A\leq C$ and $B\leq C$.

One has the representation of $K$ as the union of maximal upward
directed sets $K_z$ indexed by the points of the unit circle $S^1$,
that is, $K=\bigcup\limits_{z\in S^1}K_z$, where $z=e^{2\pi ix}, \
x\in [0,1),$ and
$$
K_z:=\left\{A\in K \ \big | \ A\subset S^1\setminus\{z\}\right\}.
$$

A base $\{U_A \ | \ A\in K\}$ for the topology $\tau$ on the index
set $I=S^1$ consists of the sets
$$
U_A=\left\{ z\in S^1 \ \big | \ A\in K_z\right\}=\left\{ z \ \big |
\ z\in S^1\setminus A\right\},
$$
that is, $U_A=S^1\setminus A$. Thus, the elements of the base for
the topology $\tau$ are all open arcs of the unit circle $S^1$:
\begin{multline*}
\left\{B\subset S^1 \ \big | \ B=(e^{2\pi ix},e^{2\pi iy}) \ \
\mbox{\rm or } \ A=S^1\setminus [e^{2\pi ix},e^{2\pi iy}], \right.
\\ \left.
\mbox{\rm where } \ x,y\in [0,1) \ \ \mbox{\rm and } \ x<y\right\}.
\end{multline*}
The topology $\tau$ coincides with the natural
topology on the unit circle $S^1$ that is compact.

\section{Main results}\label{res}

Let $K$ be an arbitrary partially ordered set. By Lemma~\ref{Kdecomposition}
we have representation (\ref{Kdecomposition}) of the set $K$ as the union of all maximal upward
directed subsets $K_i, \ i\in I$.

For each index $i\in I$ we consider an inductive system
$\mathcal{F}_i=(K_i,\{\mathcal{T}_a\},\{\textrm{id}_{ba}\})$ consisting of Toeplitz algebras, that is,
$\mathcal{T}_a=\mathcal{T}$ for all $ a\in K$,
 and the bonding $*$-homomorphisms $\textrm{i}_{ba}:\mathcal{T}_a\longrightarrow \mathcal{T}_b$,
 where $a\leq b$,  are the identity mappings.

Let us construct the inductive limits of the above-mentioned inductive systems:
$$(\mathcal{T}^{K_i},\{\textrm{id}^{K_i}_a\}):=\varinjlim\mathcal{F}_i=\varinjlim(K_i,\{\mathcal{T}_a\},\{\textrm{id}_{ba}\}).$$
It is clear that one has the isomorphism $\mathcal{T}^{K_i}\simeq\mathcal{T}$ of $C^*$-algebras.

Further, we take any element $a\in K$ and consider the direct
product of $C^*$-algebras
\begin{multline*}
\mathfrak{B}_a:=\prod\limits_{i\in U_a}\mathcal{T}^{K_i}=
\left\{f:U_a\longrightarrow  \bigcup_{i\in U_a}\mathcal{T}^{K_i}:i
\mapsto f(i) \in  \mathcal{T}^{K_i} \, \Big | \right. \\
\left. \|f\|=\sup_i\|f(i)\|< +\infty \right\}.
\end{multline*}

For every pair of elements $a,b\in K$ such that the inclusion $U_b\subset U_a$ holds, we define
the $*$-homomorphism
$\tau_{ba}:\mathfrak{B}_a\longrightarrow\mathfrak{B}_b$ by the~rule:
$$
\tau_{ba}(f)(j)=f(j),
$$
where $f\in \mathfrak{B}_a$ and $j\in U_b$. Obviously, one has the
equality $\tau_{ca}=\tau_{cb}\circ\tau_{ba}$ whenever $a,b,c\in K$
and  the condition \ $U_c\subset U_b\subset U_a$ \ holds.

Therefore, for each index $i\in I$ we can consider the inductive
system $(K_i,\{\mathfrak{B}_a\},\{\tau_{ba}\})$. Here,
the bonding $*$-homomorphisms $\tau_{ba}:\mathfrak{B}_a\longrightarrow\mathfrak{B}_b$ are defined
for all pairs of elements $a,b\in K$ satisfying the condition $U_b\subset U_a$, in particular, whenever  $a\leq b$. The  inductive limit of this system is denoted by
$$(\mathfrak{B}^{K_i}, \{\tau^{K_i}_a\} ):=\varinjlim(K_i,\{\mathfrak{B}_a\},\{\tau_{ba}\}).$$

We note that the analog of equality~
(\ref{sigma}) \, is valid, that is,  $\tau^{K_i}_a=\tau^{K_i}_b\circ\tau_{ba}$ whenever $a\leq b$.

\begin{lemma}\label{lim}
Let $i\in I$ be a non-isolated point with a countable neighbourhood base
 $\{U_{a_n}\mid a_n\in K_i, n\in \mathbb{N}\}$ satisfying the condition
\  $U_{a_1}\supset U_{a_2}\supset
U_{a_3}\supset...$. \ Let $$(\mathfrak{B},
\{\tau_{n}\}):=\varinjlim(\{\mathfrak{B}_{a_n}\},\{\tau_{a_{n+1}a_n}\})$$
be the inductive limit of the inductive sequence
$$
\xymatrix{ \mathfrak{B}_{a_1} \ar[r]^{\tau_{a_2a_1}} &
\mathfrak{B}_{a_2} \ar[r]^{\tau_{a_3a_2}} & \mathfrak{B}_{a_3}
\ar[r]^{\tau_{a_4a_3}} & ...},
$$
where $ \tau_{a_{n+1}a_n}(f)(j)=f(j)$ for $f\in \mathfrak{B}_{a_n}$
and $j\in U_{a_{n+1}}$. Then there exists an isomorphism of $C^*$-algebras
\begin{equation}\label{isom}
\mathfrak{B}\simeq\mathfrak{B}^{K_i}.
\end{equation}
\end{lemma}
\begin{proof} Using the universal property for the inductive limits (see Lemma \ref{UniversalProp}), we have the unique $*$-homomorphism $\Psi:
\mathfrak{B}\longrightarrow\mathfrak{B}^{K_i}$ such that the following diagram is commutative:
\begin{equation}\label{diagram1}
\xymatrix{ \mathfrak{B}_{a_n} \ar[rr]^{\tau_{a_{n+1}a_n}}
\ar[dr]^{\tau_n} \ar[ddr]_{\tau^{K_i}_{a_n}} & &
\mathfrak{B}_{a_{n+1}} \ar[dl]_{\tau_{n+1}}
\ar[ddl]^{\tau^{K_i}_{a_{n+1}}} \\
&  \mathfrak{B} \ar[d]^{\Psi} & \\
&  \mathfrak{B}^{K_i}     }
\end{equation}

Take elements \ $a,b\in K_i$ \ for which the inclusion $U_b\subset U_a $ holds. Since the family
$\{U_{a_n}\mid a_n\in K_i, n\in \mathbb{N}\}$ is a neighbourhood base at the point
$i$ there exists a number $n\in\mathbb{N}$ such that we have the inclusions
$U_{a_n}\subset U_b\subset U_a $. The equality
$\tau_{a_na}=\tau_{a_nb}\circ\tau_{ba}$ implies the commutativity of the diagram
\[
\xymatrix{
  \mathfrak{B}_a \ar[rr]^{\tau_{ba}} \ar[dr]_{\tau_n\circ\tau_{a_na}}
                &  &    \mathfrak{B}_b \ar[dl]^{\tau_n\circ\tau_{a_nb}}    \\
                & \mathfrak{B}                 }
\]
Using again the universal property (Lemma
\ref{UniversalProp}), we get the unique
$*$-homomorphism $\Phi:
\mathfrak{B}^{K_i}\longrightarrow\mathfrak{B}$ making the following diagram commute:
\begin{equation}\label{diagram2}
\xymatrix{ \mathfrak{B}_{a} \ar[rr]^{\tau_{ba}}
\ar[dr]^{\tau^{K_i}_a} \ar[ddr]_{\tau_n\circ\tau_{a_na}} & &
\mathfrak{B}_{b} \ar[dl]_{\tau^{K_i}_{b}}
\ar[ddl]^{\tau_n\circ\tau_{a_{n}b}} \\
&  \mathfrak{B}^{K_i} \ar[d]^{\Phi} & \\
&  \mathfrak{B}     }
\end{equation}

By the universal property of the inductive limit, making use of diagrams
(\ref{diagram1}) and (\ref{diagram2}), we obtain the equalities
$\Psi\circ\Phi=\textrm{id}$ and $\Phi\circ\Psi=\textrm{id}$. This
means that we have isomorphism (\ref{isom}), as required.
\end{proof}

\begin{theorem}\label{theorem} Let $i\in I$ be a non-isolated point with a countable neighbourhood base.
Then for every sequence of prime numbers $P=(p_1,p_2,p_3,...)$ there exists
an injective $*$-homomorphism of $C^*$-algebras:
$$C_r^*(Q_P^+)\longrightarrow \mathfrak{B}^{K_i}.$$
\end{theorem}
\begin{proof} We take a countable neighbourhood base
$\{U_{a_n}\mid a_n\in K_i, n\in \mathbb{N}\}$ at the point $i$ satisfying the conditions
$$U_{a_1}\supset U_{a_2}\supset U_{a_3}\supset...$$ and $a_1\leq a_2\leq a_3\leq
...$ . Consider the inductive sequence
$(\{\mathfrak{B}_{a_n}\},\{\tau_{a_{n+1}a_n}\})$, where $
\tau_{a_{n+1}a_n}(f)(j)=f(j)$ for $f\in \mathfrak{B}_{a_n}$,
$j\in U_{a_{n+1}}$, and its inductive limit
$$(\mathfrak{B},
\{\tau_{n}\}):=\varinjlim(\{\mathfrak{B}_{a_n}\},\{\tau_{a_{n+1}a_n}\}).$$

For every $n\in\mathbb{N}$ we define $W_n:=U_{a_n}\setminus
U_{a_{n+1}}$ and the operator-valued function $L_{a_n}$ in the algebra
$\mathfrak{B}_{a_n}$ as follows. For an index $j\in U_{a_n}$
we put
\begin{equation}\label{Lan}
L_{a_n}(j)=T^{p_np_{n+1}...p_k}, \textrm{
if } j\in W_k, \ k\in \mathbb{N}, \ k\geq n.
\end{equation}
Recall that here $T$ is the shift operator generating the Toeplitz
algebra $\mathcal{T}.$

Now we set
\begin{equation}\label{Ln}
L_n:=\tau_{n}(L_{a_n}).
\end{equation}
Let us show that  $L_n$ is an isometry in $C^*$-algebra
$\mathfrak{B}$. Indeed, since $\tau_{n}$ is a unital \,
$*$-homomorphism, one has the equalities
$$L_n^*L_n=\tau_{n}(L_{a_n}^*L_{a_n})=\tau_{n}(\mathbb{I}_{\mathfrak{B}_{a_n}})=
\mathbb{I}_{\mathfrak{B}}.$$

Further, we consider the inductive sequence of Toeplitz algebras
$(\{\mathcal{T}_n\},\{\varphi_{n,n+1}\}),$ where
$\mathcal{T}_n=\mathcal{T}$ and the bonding $*$-homomorphisms
$\varphi_{n,n+1}$ are defined by
\begin{equation}\label{phi}
\varphi_{n,n+1}:\mathcal{T}_n\longrightarrow\mathcal{T}_{n+1}:T\mapsto
T^{p_n}, \ n\in \mathbb{N}.
\end{equation}
By Lemma~\ref{gumerov}, the inductive limit of this sequence is isomorphic to the reduced semigroup $C^*$-algebra $C^*_r(Q^+_P)$. Hence, there exist injective
$*$-homomorphisms $\varphi_n:\mathcal{T}_n\longrightarrow
C^*_r(Q^+_P)$ such that the equality
$$\varphi_{n+1}\circ\varphi_{n,n+1}=\varphi_n$$
holds for every $n\in\mathbb{N}$.
This means that the diagram
\[
\xymatrix{
  \mathcal{T}_n \ar[rr]^{\varphi_{n,n+1}} \ar[dr]_{\varphi_n}
                &  &    \mathcal{T}_{n+1} \ar[dl]^{\varphi_{n+1}}    \\
                & C^*_r(Q^+_P)                 }
\]
is commutative.

It follows from Coburn's theorem \cite[Theorem 3.5.18]{murphy}
that for each number $n\in\mathbb{N}$ there is a unique isometric
 $*$-homomorphism
$\psi_n:\mathcal{T}_n\longrightarrow\mathfrak{B}$ such that the condition
\begin{equation}\label{psiL}\psi_n(T)=L_n.\end{equation}
holds.

We claim that the diagram
\[
\xymatrix{
  \mathcal{T}_n \ar[rr]^{\varphi_{n,n+1}} \ar[dr]_{\psi_n}
                &  &    \mathcal{T}_{n+1} \ar[dl]^{\psi_{n+1}}    \\
                & \mathfrak{B}                }
\]
is commutative, that is, the following equality for $*$-homomorphisms holds:
\begin{equation}\label{psi}
\psi_{n+1}\circ\varphi_{n,n+1}=\psi_n.
\end{equation}
Really, to prove equality (\ref{psi}) it is enough to show that for  the homomorphisms $\psi_{n+1}\circ\varphi_{n,n+1}$ and $\psi_n$  their values at the generating element
$T$ for the Toeplitz algebra $\mathcal{T}_n$ are the same. Making use of (\ref{phi}),
(\ref{psiL}) and (\ref{Ln}), we obtain the equalities
\begin{equation}\label{psileft}
(\psi_{n+1}\circ\varphi_{n,n+1})(T)=(L_{n+1})^{p_n}=\tau_{n+1}(L_{a_{n+1}}^{p_n}).
\end{equation}
On the other hand, using (\ref{psiL}), (\ref{Ln}) and the definition of the inductive limit, we get
\begin{equation}\label{psiright}
\psi_n(T)=\tau_{n}(L_{a_n})=(\tau_{n+1}\circ\tau_{a_{n+1}a_n})(L_{a_n}).
\end{equation}

Now we shall show that in the algebra $\mathfrak{B}_{a_{n+1}}$ the following equality holds:
\begin{equation}\label{tauL}
L_{a_{n+1}}^{p_n}=\tau_{a_{n+1}a_n}(L_{a_n}).
\end{equation}
Indeed, firstly, by (\ref{Lan}), we can write
$$L_{a_{n+1}}^{p_n}(j)=(T^{p_{n+1}...p_k})^{p_n}=T^{p_np_{n+1}...p_k},
\textrm{ if } j\in W_k, \ k\in \mathbb{N}, \ k\geq n+1.$$ Secondly,
we have the equalities
$$\tau_{a_{n+1}a_n}(L_{a_{n}})(j)=L_{a_{n}}(j), \textrm{ if } j\in
U_{a_{n+1}}$$ and
$$L_{a_{n}}(j)=T^{p_np_{n+1}...p_k}, \textrm{ if } j\in W_k, \
k\in \mathbb{N}, \ k\geq n.$$ This means that for every index $j\in
U_{a_{n+1}}$ we get the equality
$(L_{a_{n+1}}^{p_n})(j)=(\tau_{a_{n+1}a_n}(L_{a_n}))(j)$.
Consequently, equality (\ref{tauL}) is true. Thus, the expressions
on the right-hand sides in (\ref{psileft}) and (\ref{psiright}) are
equal. Therefore, equality (\ref{psi}) is valid, as claimed.

By the universal property for the inductive limits (see Lemma
\ref{UniversalProp}), there exists a unique $*$-homomorphism
$\theta: C^*_r(Q^+_P)\longrightarrow\mathfrak{B}$, such that the
following diagram is commutative:
$$
\xymatrix{ \mathcal{T}_{n} \ar[rr]^{\varphi_{n,n+1}}
\ar[dr]^{\varphi_n} \ar[ddr]_{\psi_{n}} & & \mathcal{T}_{n+1}
\ar[dl]_{\varphi_{n+1}}
\ar[ddl]^{\psi_{n+1}} \\
&  C^*_r(Q^+_P) \ar[d]^{\theta} & \\
&  \mathfrak{B}     }
$$
Since all the mappings
$\psi_n:\mathcal{T}_n\longrightarrow\mathfrak{B}$ are injective
$*$-homomorphisms, the  $*$-homomorphism  $\theta:
C^*_r(Q^+_P)\longrightarrow\mathfrak{B}$ is also injective (see
Lemma \ref{injective}).

Finally, to complete the proof of the theorem we use Lemma~\ref{lim} which states that the $C^*$-algebras
$\mathfrak{B}$ and $\mathfrak{B}^{K_i}$ are isomorphic.

\end{proof}

 Let us consider a sequence of prime numbers $P$ such that each prime number from $\mathbb{N}$ is equal to infinitely many terms of $P$, for example, $P=(2,2,3,2,3,5,2,3,5,7,\ldots)$.
It is straightforward to check that  the following equality holds for semigroups of rational numbers:
 $$
 \mathbb{Q}^+_P=\mathbb{Q}^+ :=\mathbb{Q}\cap[0,+\infty).$$

 As a consequence of Theorem~\ref{theorem}, we obtain

\begin{theorem}
Let $i\in I$ be a non-isolated point with a countable neighbourhood
base.  There exists an injective $*$-homomorphism of $C^*$-algebras:
$$C_r^*(Q^+)\longrightarrow \mathfrak{B}^{K_i}.$$
\end{theorem}

%
%

\

\end{document}